\documentclass[11pt, oneside]{amsart}
\usepackage{amsmath}
\usepackage{amsthm}
\usepackage{amssymb} 
\usepackage[mathscr]{euscript}
\usepackage{amsmath}
\usepackage{amsthm}
\usepackage{amssymb}
\usepackage{fancyhdr}
\usepackage{enumerate}
\usepackage{amscd}
\usepackage[all]{xy}
\usepackage{graphicx}
\usepackage{color}
\usepackage{hyperref}
\hypersetup{
    colorlinks,
    citecolor=black,
    filecolor=black,
    linkcolor=black,
    urlcolor=black
}
\usepackage{tikz}
\usepackage[all]{xy}
\usepackage{graphicx}
\usetikzlibrary{backgrounds}
\usepackage{comment}

\theoremstyle{plain}
\newtheorem{lemma}{Lemma}[section]

\newtheorem{proposition}[lemma]{Proposition}

\newtheorem{theorem}[lemma]{Theorem}
\newtheorem{Theorem}{Theorem}
\newtheorem{corollary}[lemma]{Corollary}

\newtheorem*{main-theorem}{Theorem~\ref{theorem:main}}
\numberwithin{equation}{section}

\theoremstyle{remark}

\theoremstyle{definition}
\newtheorem{remark}[lemma]{Remark}
\newtheorem{example}[lemma]{Example}
\newtheorem{definition}[lemma]{Definition}
\newtheorem*{notation}{Notation}
\newtheorem*{ack}{Acknowledgement}

\usepackage[tmargin=1in,bmargin=1in,lmargin=1in,rmargin=1in]{geometry}
\begin{document}
\title{Mass rigidity for hyperbolic manifolds}
\author{Lan-Hsuan Huang}
\author{Hyun Chul Jang}
\address{Department of Mathematics, University of Connecticut, Storrs, CT 06269, USA}
\email{lan-hsuan.huang@uconn.edu\\ hyun.c.jang@uconn.edu}

\author{Daniel Martin}
\address{Department of Mathematics, Trinity College, Hartford, CT 06106, USA}
\email{daniel.martin@trincoll.edu}
\date{}
\dedicatory{Dedicated to Professor Greg Galloway on the occasion of
  his seventieth birthday}
\begin{abstract}
We prove the rigidity of the positive mass theorem for asymptotically hyperbolic manifolds. Namely, if the mass equality $p_0=\sqrt{p_1^2+\cdots+ p_n^2}$ holds, then the manifold is isometric to hyperbolic space. The result was previously proven  for spin manifolds  \cite{Min-Oo:1989, Wang.X:2001, Andersson-Dahl:1998, Chrusciel-Herzlich:2003} or under special asymptotics~\cite{Andersson-Cai-Galloway:2008}. 

\end{abstract}
\maketitle

\section{Introduction}

One of the central topics in differential geometry is to understand how Riemannian manifolds can be characterized under a  curvature assumption. The seminal work of R. Schoen and S.-T. Yau~\cite{Schoen-Yau:1979-pmt1} of the Riemannian positive mass theorem establishes a characterization of Euclidean space. Specifically,  Euclidean space is the  unique asymptotically flat manifold with nonnegative scalar curvature that has zero ADM mass, which is an invariant defined at the manifold's infinity.    E.~Witten~\cite{Witten:1981} later introduced a different method using spinors. M.~Min-Oo~\cite{Min-Oo:1989} adapted the spinor approach for strongly asymptotically hyperbolic manifolds  (with a  corrected assumption by~E.~Delay~\cite[Definition 1]{Delay:1997hy}) and gave a characterization of hyperbolic space, which was later refined by L.~Andersson and M.~Dahl~\cite{Andersson-Dahl:1998}.  Based on the spinor approach, X. Wang~\cite{Wang.X:2001} defined the mass and  established the positive mass theorem for conformally compact, asymptotically hyperbolic manifolds  $(X^n, g)$ whose conformal boundary  is the unit round sphere $(S^{n-1}, h)$ and with the following expansion:
\begin{align} \label{equation:conformal-compact}
	g = \frac{1}{(\sinh \rho)^{2}} \left( d\rho^2 + h + \frac{\rho^n}{n} \kappa + O(\rho^{n+1})\right)
\end{align}
where $\rho$ is a boundary defining function and $\kappa$ is a symmetric $(0, 2)$-tensor defined on $S^{n-1}$. The mass $(p_0, p_1,\dots, p_n)$ of $g$ is defined by 
\[
p_0 = \int_{S^{n-1}} \mathrm{tr}_h \kappa\, d\mu_h, \qquad p_i = \int_{S^{n-1}}x_i \mathrm{tr}_h \kappa \, d\mu_h \qquad \mbox{for $i=1, \dots, n$}
\]
where $(x_1, \dots, x_n)$  are the Cartesian coordinates of $\mathbb{R}^n$ restricted on $S^{n-1}$. It is  an intriguing  observation that the mass consists of $(n+1)$ numbers $(p_0, p_1, \dots, p_n)$, instead of a single number, the ADM mass, as for the asymptotically flat manifolds.  In \cite{Chrusciel-Herzlich:2003}, P.~Chru\'sciel and M. Herzlich extended the definition of mass to a larger class of manifolds without assuming conformal compactification and obtained a flux integral formula, which we will recall in Definition~\ref{def:massfunctional}.  As a result, the following positive mass theorem holds for spin manifolds. 
\begin{Theorem}[\cite{Wang.X:2001, Chrusciel-Herzlich:2003}]
Let $n\ge 3$ and $(X, g)$ an $n$-dimensional asymptotically hyperbolic manifold with scalar curvature $R_g \ge -n(n-1)$. Suppose $X$ is spin. Then $p_0 \ge \sqrt{p_1^2+ \cdots+ p_n^2}$ with equality only if $(X, g)$ is isometric to hyperbolic space. 
\end{Theorem}

It has been conjectured that the positive mass theorem for asymptotically hyperbolic manifolds holds without the spin assumption. Assuming that the mass aspect function $\mathrm{tr}_h \kappa$ in \eqref{equation:conformal-compact} is either everywhere positive, zero, or negative on $S^{n-1}$,  L.~Andersson, M.~Cai, and G.~Galloway \cite{Andersson-Cai-Galloway:2008} proved the positive mass theorem for dimensions $3\le n \le 7$. For more general asymptotics, an approach using Jang's equation to the positivity of mass in three dimensions was announced by A. Sakovich. A recent paper~\cite{Chrusciel-Delay:2019}  of P. Chru\'sciel and E. Delay proves  the positivity by a gluing argument in general dimensions. Nevertheless, these two approaches to the positivity of mass are indirect and do not seem to give information about the equality case, which is the focus of the current paper.

Our main result is the following rigidity statement. The technical terms are defined in Section~\ref{section:preliminary}. We also refer to Definition~\ref{definition:AH} for the precise definition of asymptotically hyperbolic manifolds, which includes a technical assumption that $g\in C^\infty_\mathrm{loc}$. See Remark~\ref{remark:smooth}.

\begin{Theorem}\label{theorem:main}
Let $n\ge 3$ and $(M,g)$ be an $n$-dimensional asymptotically hyperbolic manifold with scalar curvature $R_g \ge -n(n-1)$ and with equality $p_0 = \sqrt{p_{1}^2 + \cdots + p_{n}^2}$, where $(p_0, p_1, \dots, p_n)$ is the mass of $g$.   Suppose the following holds:
\begin{itemize}  
\item[($\star$)] There is an open neighborhood $\mathcal{M}$ of $g$ in the space of asymptotically hyperbolic metrics on $M$ such that the inequality $p_0(\gamma)\ge \sqrt{ (p_1(\gamma))^2 + \dots + (p_n(\gamma))^2}$ holds if $\gamma\in \mathcal{M}$ and the scalar curvature satisfies $R_\gamma = R_g$.
\end{itemize}
 Then $(M, g)$ is isometric to hyperbolic space.
 \end{Theorem}

Using  positivity of mass proven in \cite{Chrusciel-Delay:2019},  the assumption~($\star$)  can be dropped and thus we arrive at the following result. 

\begin{Theorem}
Let $n\ge 3$ and $(M,g)$ an $n$-dimensional asymptotically hyperbolic manifold with scalar curvature $R_g \ge -n(n-1)$ and with the equality $p_0 = \sqrt{p_{1}^2 + \cdots + p_{n}^2}$.  Then $(M, g)$ is isometric to hyperbolic space.
\end{Theorem}

We outline the proof of Theorem~\ref{theorem:main}, which is included in Section~\ref{section:main-theorem}. We show that a metric that realizes the mass equality is a minimizer of a functional~$\mathcal{F}$, defined by \eqref{equation:functional}, subject to a  scalar curvature constraint. By studying the first variation of this functional, we show that such a metric must be static and, in fact, possess a static potential with certain asymptotics. The desired characterization of hyperbolic space follows from proving a static uniqueness result. 

We remark that the approach is motivated by a constrained minimization scheme proposed by R.~Bartnik~\cite{Bartnik:2005} for his quasi-local mass program. The connection between the constrained minimization and mass rigidity was recently employed by D. Lee and the first named author in their proof to the rigidity conjecture of the spacetime positive mass theorem~\cite{Huang-Lee:2019}. 

In our proof of Theorem~\ref{theorem:main}, it is essential to analyze the scalar curvature map and to derive the following result.
\begin{Theorem}\label{theorem:surjectivity}
Let $(M, g)$ be an $n$-dimensional asymptotically hyperbolic manifold. For $k\ge 2$ and $s\in (-1, n)$, the linearized scalar curvature map
\[
	L_g : C^{k,\alpha}_{-s}(M)\to C^{k-2,\alpha}_{-s}(M)
\]
is surjective. As a consequence, the scalar curvature map is locally surjective at $g$. Namely, there are constants $\epsilon, C>0$ such that if $\| \phi - R_g \|_{C^{k-2,\alpha}_{-s}(M)} <\epsilon$, then there is a metric $\gamma$ with $\| \gamma-g\|_{C^{k,\alpha}_{-s}(M)}\le C\epsilon$  that realizes the scalar curvature $R_\gamma = R_g + \phi$. 
\end{Theorem}
Theorem~\ref{theorem:surjectivity} is also of independent interest from the perspective of scalar curvature deformation. For example, it produces infinitely many asymptotically hyperbolic metrics  with scalar curvature greater than $-n(n-1)$ by perturbation. 

We remark that  the weighted H\"older space is chosen as our analytical framework because the known results on the positivity of mass require that regularity.  It is shown that  the Einstein constraint map is surjective among the appropriate weighted \emph{Sobolev} spaces by E.~Delay and J.~Fougeirol \cite{Delay-Fougeirol:2016}. However, it does not seem to imply Theorem~\ref{theorem:surjectivity}. In fact, our proof relies on a different argument. One difficulty is that  the dual space $(C^{k-2,\alpha}_{-s})^*$ is not well-understood.  Efforts are made to analyze the kernel of the adjoint operator $L_g^*$ on $(C^{k-2,\alpha}_{-s})^*$ without assuming the kernel elements to decay at infinity. See Section~\ref{section:static} and more specifically, Theorem~\ref{theorem:expansion}.

Finally, we remark that the proof of Theorem~\ref{theorem:surjectivity} uses the assumption that an asymptotically hyperbolic manifold  is complete without boundary (see Definition~\ref{definition:AH}). For manifolds with  compact boundary,  while the same argument still works if one imposes either Dirichlet or Neumann type  condition on  the  metrics, we need the surjectivity to hold for tensors with stronger vanishing condition at the boundary to establish the mass rigidity. In a forthcoming paper, we use a different argument and extend Theorem~\ref{theorem:surjectivity} for metrics that coincide with $g$ of infinite order at the boundary. It enables us to  prove the mass rigidity for asymptotically \emph{locally} hyperbolic manifolds. In that setting, the model spaces that we consider have compact boundary with natural geometric boundary conditions.

\begin{ack}
The project was initiated while the authors participated in the 2017 summer program on  Geometry and Relativity at the Erwin Schr\"odinger Institute. We would like to express our sincere gratitude to the organizers Robert Beig, Piotr Chru\'sciel, Michael Eichmair, Greg Galloway, Richard Schoen, Tim-Torben Paetz for their warm hospitality and the inspiring program.  

The project was  partially supported by the NSF Career award DMS-1452477. L.-H.~Huang was also  supported by Simons Fellowship of the Simons Foundation and  von Neumann Fellowship at the Institute for Advanced Study.
\end{ack}

\section{Preliminaries}\label{section:preliminary}

\subsection{Weighted H\"older spaces and asymptotically hyperbolic manifolds}

Denote by $\mathbb{H}^n$ the $n$-dimensional hyperbolic space with scalar curvature $-n(n-1)$. As our model for hyperbolic space, we consider the upper-sheet of the hyperboloid  in Minkowski space $(\mathbb{R}^{n,1}, -dt^2 + dx_1^2 + \cdots + dx_n^2$), defined by
\[
	\mathbb{H}^n = \Big\{ (x,t) = (x_1, \dots, x_n, t) \in \mathbb{R}^{n,1}: t = \sqrt{1+x_1^2+ \dots + x_n^2}\Big\}.
\] 
The restriction of the Minkowski metric to the upper-sheet hyperboloid is hyperbolic space and can be expressed in spherical coordinates as 
\begin{align}\label{equation:hyperboloid}
	b= \frac{1}{1+r^2}\, dr^2+ r^2 h,
\end{align}
where $r = |x|:= \sqrt{x_1^2+\cdots + x_n^2}$ is the radial coordinate, and $h$ is the standard metric on the round unit $(n-1)$-sphere. We refer $(\mathbb{R}^n, b)$ as the \emph{hyperboloid model} of hyperbolic space.

The volume form of $b$ is $d\mu_b = \frac{r^{n-1}}{\sqrt{1+r^2}}\, dr \,d\omega$, where $d\omega $ is the volume form on the round unit $(n-1)$-sphere.  By the co-area formula, it is direct to see that the induced volume form on $S_r=\{ |x|=r\}$ of the hyperbolic metric $b$ is the same as the standard volume form on the round $(n-1)$-sphere of radius~$r$.

Let $B$ be an open ball in $\mathbb{R}^n$ centered at the origin. Denote $\mathbb{H}^n\setminus B = (\mathbb{R}^n\setminus B, b)$.   We fix an orthonormal frame $\{ e_1, \dots, e_n\}$ on $\mathbb{H}^n\setminus B$ defined by, with respect to the spherical coordinates $\{ r, \theta_1, \dots, \theta_{n-1}\}$, 
\begin{align} \label{equation:orthonormal-frame}
	e_1 = \sqrt{1+r^2} \tfrac{\partial}{\partial r}, \quad e_2 = r^{-1} \tfrac{\partial}{\partial \theta_1},\quad \dots,\quad  e_n = (r\sin(\theta_1)\dots\sin(\theta_{n-2}))^{-1} \tfrac{\partial}{\partial \theta_{n-1}}. 
\end{align}

\begin{definition}
For $k=0, 1, 2, \dots$, $ \alpha \in (0, 1)$, and $q\in \mathbb{R}$, we define the \emph{weighted H\"older spaces} $C^{k,\alpha}_{-q} (\mathbb{H}^n\setminus B)$ as the collection of $C^{k,\alpha}_{\mathrm{loc}}(\mathbb{H}^n\setminus B)$ functions $f$ on $\mathbb{H}^n\setminus B$ that satisfy
\begin{align*}
	\| f \|_{C^{k,\alpha}_{-q}(\mathbb{H}^n\setminus B)} := \sum_{\ell =0, 1,\dots, k}\sup_{x\in \mathbb{H}^n\setminus B} |x|^q |\mathring{\nabla}^\ell f(x)|_b +\sup_{x\in \mathbb{H}^n\setminus B} |x|^{q} [\mathring{\nabla}^k f]_{\alpha; B_1(x)} <\infty,
\end{align*}
where $\mathring{\nabla}$ is the covariant derivative with respect to $b$,  
\[
	[\mathring{\nabla}^k f]_{\alpha; B_1(x)} =\sup_{1\le i_1, \dots, i_k \le n} \sup_{y\neq z\in B_1(x) } \frac{|e_{i_1} \cdots e_{i_k} (f)(y) - e_{i_1} \cdots e_{ik} (f) (z) |}{(d_b(y,z))^\alpha},
\]
and $B_1(x)$ is the unit ball centered at $x$ intersecting with $\mathbb{H}^n\setminus B$. We extend the definition to tensors of arbitrary types:  a tensor $h\in C^{k,\alpha}_{-q}(\mathbb{H}^n \setminus B)$ if and only if each tensor component with respect to the orthonormal frame lies in $C^{k,\alpha}_{-q}(\mathbb{H}^n \setminus B)$.

Let $M$ be a smooth manifold covered by an atlas that consists of a non-compact chart $\Phi: M\setminus K \cong \mathbb{H}^n\setminus B$ and finitely many compact charts.   We define the \emph{weighted H\"older norm} $\| f \|_{C^{k,\alpha}_{-q}(M)}$ (for a function or tensor) to be the sum of the weighted norm $\| \Phi_* f\|_{C^{k,\alpha}_{-q}(\mathbb{H}^n\setminus B)} $ and the usual $C^{k,\alpha}$ norms on  compact charts. Denote by $C^{k,\alpha}_{-q}(M)$ the completion of $C^{k,\alpha}_{c}(M)$ functions with respect to the  weighted H\"older norm. We often suppress $M$ when the context is clear.  
\end{definition}
\begin{notation}
Throughout the paper, we use the notation $O^{k,\alpha}(r^{-q})$ to denote a function or tensor, that belongs to the corresponding weighted space $C^{k,\alpha}_{-q} (M)$.  We simply write $O(r^{-q})$ in place of $O^0 (r^{-q})$.
\end{notation}

We collect the following basic facts about the weighted H\"older spaces. 
\begin{lemma}\label{lemma:basic}
Let $k= 0, 1, 2, \dots$, $\alpha \in (0, 1)$, and $q, s\in \mathbb{R}$.
\begin{enumerate}
\item  $|x|^{-q}\in C^{k,\alpha}_{-q}(M\setminus K)$.
\item  $f\in C^{k,\alpha}_{-q}(M\setminus K)$ if and only if $|x|^sf \in C^{k,\alpha}_{s-q}(M\setminus K)$.
\item If $f\in C^{k,\alpha}_{-s}, g\in C^{k,\alpha}_{-q}$, then $fg\in C^{k,\alpha}_{-s-q}$ and there is a constant $C>0$ such that 
\[
	\| f g\|_{C^{k,\alpha}_{-s-q}} \le C \| f\|_{C^{k,\alpha}_{-s}} \| g \|_{C^{k,\alpha}_{-q}}. 
\]
\item  The inclusion $C^{k, \alpha}_{-s}(M) \subset C^{k,\beta}_{-s+\epsilon}(M)$ is compact for any $\epsilon>0$ and $\beta<\alpha$. 
\end{enumerate}
\end{lemma}
\begin{proof}
The first three statements follow directly from the definition. The last statement is standard compact embedding for weighted norms. While similar statements can be found in \cite[Lemma 3.6]{Lee:2006} and \cite[Proposition 8]{Delay:1997hy}, we include the proof for completeness as the weighted norms are defined with slight variations in the literature. Let $\{ u_i\}$ be a sequence of functions in $C^{k,\alpha}_{-s}$ with $\| u_i \|_{C^{k,\alpha}_{-s} }=1$. Applying Arzela-Ascoli on a sequence of compact sets that exhaust $M$ and by a diagonal sequence argument, there is a subsequence of $\{ u_i \}$ (which we still denote by $\{ u_i\}$, without loss of generality) and a function $u\in C^{k,\alpha}_{\mathrm{loc}}$ so that $u_i$ converges to $u$ locally uniformly in $C^{k,\beta}$. That is, for $\epsilon>0$ and a compact subset $\Omega$, there is an integer $I$ (depending on $\epsilon$ and $\Omega$) such that $\| u - u_i \|_{C^{k,\beta}(\Omega)}< \epsilon$ for all $i\ge I$. In fact, $u\in C^{k,\beta}_{-s}$ because, for each compact set $\Omega$, 
\[
	\| u \|_{C^{k,\beta}_{-s} (\Omega)}  = \lim_{i\to \infty} \| u_i \|_{C^{k,\beta}_{-s}(\Omega)} \le 1.
\]	
Let $B_r$ be the coordinate ball of radius $r$.  Using $\| u_i -u\|_{C^{k,\beta}_{-s+\epsilon}(M\setminus B_r)}\le r^{-\epsilon} (\| u_i \|_{C^{k,\beta}_{-s}(M)} + \| u \|_{C^{k,\beta}_{-s}(M)} ) $, we have that $u_i $ converges to $u$ in $C^{k,\beta}_{-s+\epsilon}(M)$. 
\end{proof}

\begin{definition}\label{definition:AH}
Let $n\ge 3$ and $q\in \big(\frac{n}{2},  n\big)$. Let $M$ be an $n$-dimensional, connected, complete manifold without boundary endowed with a Riemannian metric $g \in C^\infty_{\mathrm{loc}}$. We say that $(M, g)$ is \emph{asymptotically hyperbolic} (of order $q$) if the following holds:
\begin{enumerate}
\item There exists a diffeomorphism $M\setminus K \cong  \mathbb{H}^n\setminus B$ for some compact subset $K\subset M$. We call the  induced coordinate chart as \emph{the chart at infinity}. 
\item  With respect to the chart at infinity, $g -b \in C^{2,\alpha}_{-q} (M\setminus K)$.
 \item The scalar curvature satisfies $R_g+n(n-1) \in C^{0,\alpha}_{-n-\epsilon} (M)$ for some $\epsilon>0$. \label{item:scalar}
\end{enumerate}
\end{definition}

\begin{remark}
By direct computation, the assumption (2)  implies that the Ricci curvature of $g$ satisfies $\mathrm{Ric}_g = -(n-1)g + O^{0,\alpha}(r^{-q})$.
\end{remark}

\begin{remark}\label{remark:smooth}
Note the assumption $g \in C^\infty_{\mathrm{loc}}$. We add this technical assumption to employ elliptic interior regularity for distribution solutions. Namely, if a distribution solution $u$ weakly solves $a_{ij} \partial^2_{ij} + b_i \partial_i u + cu = f$ and if $f\in C^{k-2,\alpha}_{\mathrm{loc}}$, then $u\in C^{k,\alpha}_{\mathrm{loc}}$, provided that the coefficients $a_{ij}, b_i, c$ are locally smooth. This elliptic regularity  is only used in the proofs of Theorem~\ref{theorem:surjectivity} and Theorem~\ref{theorem:mass-rigidity}. If the regularity statement holds for coefficients that are just H\"older regular, then that technical assumption may be dropped.
\end{remark}

To compare Definition~\ref{definition:AH} with various notions of asymptotically hyperbolic manifolds in the existing literature, we express the assumption (2) in Definition~\ref{definition:AH} in coordinates. It appears that our asymptotic assumption is more general than \eqref{equation:conformal-compact}.

\begin{lemma}
A $(0,2)$-tensor  $g$ satisfies  $g -b \in C^{2,\alpha}_{-q} (M\setminus K)$ if and only if the tensor components have the following asymptotics in spherical coordinates:
\begin{align*}
	g= \bigg(\frac{1}{1+r^2} + O^{2,\alpha}(r^{-2-q}) \bigg)\, dr^2+  O^{2,\alpha}(r^{-q}) dr d\theta_j + (r^2 h_{j\ell} + O^{2,\alpha}(r^{2-q}) ) d\theta_j d\theta_\ell \quad \mbox{ as $r\to \infty$}.
\end{align*}
By changing the coordinate $r= \frac{1}{\sinh \rho}$, we can  express $g$ as 
\[
	g = \frac{1}{(\sinh \rho)^2}\left[ (1+O(\rho^q ) )d\rho^2+  O(\rho^q) d\rho d\theta_i + (h_{j\ell} + O(\rho^q)) d\theta_j d\theta_\ell \right]  \quad \mbox{ as $\rho \to 0$},
\]
where we slightly abuse the $O$-notation in the previous expression and write  $u =O(\rho^q )$ if $\frac{u}{\rho^q}$ is bounded as $\rho \to 0$.
\end{lemma}
\begin{proof}
Via the diffeomorphism on the chart at infinity, it suffices to prove the result for tensors defined on $\mathbb{H}^n\setminus B$. Express $g$ in the spherical coordinates as follows:
\begin{align} \label{equation:spherical}
	g = A\,dr^2 + 2\sum_j B_j \,dr \,d\theta_j +\sum_{j,\ell} C_{j\ell}\, d\theta_j \,d\theta_\ell. 
\end{align}
By definition, $\kappa:=g-b$ belongs to $C^{k,\alpha}_{-q}(\mathbb{H}^n \setminus B)$ if and only if each tensor component $\kappa(e_i, e_j)\in C^{k,\alpha}_{-q}(\mathbb{H}^n \setminus B)$. By \eqref{equation:orthonormal-frame} and  \eqref{equation:spherical}, we have 
\[
	\kappa(e_1, e_1) = (1+r^2) A, \quad \kappa(e_1, e_{j+1}) = \sqrt{1+r^2} r^{-1} B_j, \quad \mbox{ and } \quad \kappa(e_{j+1}, e_{\ell+1}) = r^{-2} C_{j\ell}. 
\]
Thus,  $\kappa\in C^{k,\alpha}_{-q}(\mathbb{H}^n \setminus B)$ if and only if the tensor components satisfy
\[
	A\in C^{k,\alpha}_{-2-q}, \quad B_j\in C^{k,\alpha}_{-q}, \quad \mbox{ and }\quad  C_{j\ell} \in C^{k,\alpha}_{2-q}. 
\]

\end{proof}

\subsection{Wang-Chru\'sciel-Herzlich mass, and an alternative definition}
X. Wang~\cite{Wang.X:2001} defined the mass  for asymptotically hyperbolic manifolds that are conformally compact. For the class of asymptotically hyperbolic manifolds adopted in the current paper, we use the following more general definition of P. Chru\'sciel and M.~Herzlich~\cite{Chrusciel-Herzlich:2003}.

\begin{definition}\label{def:massfunctional} 
Let $(M,g)$ be an asymptotically hyperbolic manifold.  Given a function $V\in C^1(M\setminus K)$, we define the mass integral  
\begin{align}\label{equation:mass}
	H_g(V) = \lim_{r\rightarrow\infty}\int_{S_r} \left[V\big(\mathring{\mathrm{div}}\, h -d (\mathring{\mathrm{tr}} \, h) \big)(\nu_0) + (\mathring{\mathrm{tr}} \,h ) dV (\nu_0) - h(\mathring{\nabla} V, \nu_0)\right] \, d\sigma_b,
\end{align}
where $h = g-b$, $\nu_0$ is the outward unit normal vector to $S_r = \{ |x|=r\}$,  and $\mathring{\mathrm{div}}, \mathring{\mathrm{tr}}, \mathring{\nabla}$, are all with respect to $b$. The volume form $d\sigma_b$ is the restriction of the volume form of $b$ on $S_r$. The \emph{mass of Wang-Chru\'sciel-Herzlich}  is defined by
\begin{align*}
	p_0(g) = H_g(\sqrt{1+r^2})\quad  \mbox{ and } \quad p_i(g) = H_g(x_i) \mbox{ for }  i =1, \dots, n.
\end{align*}
We may omit $g$ and simply write the mass $(p_0, p_1,\dots, p_n)$ when the context is clear. 

\end{definition}

\begin{remark}\label{remark:mass}
In the above definition, we can replace the functions $\sqrt{1+r^2}$  and $x_i$ by  $\sqrt{1+r^2} + O^2 (r^{1-q})$ and $x_i + O^2(r^{1-q})$ respectively, since the differences in the corresponding mass integrals go to zero in the limit. For the same reason, we may also replace $\nu_0, \mathring{\mathrm{div}}, \mathring{\mathrm{tr}}, \mathring{\nabla}$, and $d\sigma_b$ in \eqref{equation:mass} by the corresponding objects with respect to another asymptotically hyperbolic metric and still obtain the same limit.
\end{remark}

\begin{remark}\label{remark:static}
The quantity $(p_0, p_1, \dots, p_n)$ is a geometric invariant among an appropriate class of  charts at infinity (see \cite{Chrusciel-Herzlich:2003}, also \cite{Herzlich:2005}). We denote the functions  appearing in the above definition by 
\[
V_0= \sqrt{1+r^2} \quad \mbox{ and } \quad V_i=  x_i \quad \mbox{ for } i=1, \dots, n.
\]
In $\mathbb{H}^n$, these functions satisfy the differential equation $\mathring{\nabla}^2 V_i = V_i b$, for $i =0, 1, \dots, n$. They are the so-called  \emph{static potentials}. We will discuss general properties of static potentials in an asymptotically hyperbolic manifold in Section~\ref{section:static}. 
\end{remark}

We recall an equivalent definition of mass, which will be used in the proof of the main theorem.  This formula is  known to the experts and is stated in  \cite[Theorem 3.3]{Herzlich:2016}, whose proof is similar to the analogous formula for asymptotically flat manifolds.  
\begin{proposition}\label{proposition:mass}
Let $(M, g)$ be an asymptotically hyperbolic manifold. If $V\in C^2(M\setminus K)$ satisfies
\[
	\mathring{\nabla}^2 V = Vb,
\] 
then
\[
	\lim_{r\to \infty} \int_{S_r} (\mathrm{Ric}_g + (n-1)g)(\mathring{\nabla} V,  \nu_0) \, d\sigma_b = -\tfrac{n-2}{2} H(V),
\]
provided the quantity on either side of the equation converges. 
\end{proposition}

\subsection{Operators asymptotic to $\Delta -n$}

To analyze the scalar curvature operator on an asymptotically hyperbolic manifold, the following class of operators naturally appears. 
\begin{definition}
Let $(M, g)$ be asymptotically hyperbolic. Let $\Delta$ be the Laplace-Beltrami operator of~$g$, which is the trace of the covariant Hessian. For $k\ge 2$, we say that the differential operator $T: C^{k,\alpha}_{-s} \to C^{k-2,\alpha}_{-s}$ defined by $Tu = \Delta u + \xi \cdot \nabla u + \eta u$  is \emph{asymptotic to $\Delta - n$} if there is a number $\epsilon>0$ such that the vector field $\xi\in C^{k-2,\alpha}_{-\epsilon}$ and the function $\eta + n \in C^{k-2,\alpha}_{-\epsilon}$.   
\end{definition}

We recall the following classical result on isomorphism. 
\begin{lemma}\label{lemma:isomorphism}
Let $(M, g)$ be an $n$-dimensional asymptotically hyperbolic manifold and $s\in (-1, n)$. The operator $T_0 : C^{k,\alpha}_{-s}(M)\to C^{k-2,\alpha}_{-s}(M)$ defined by $T_0  u = \Delta u - n u$ is an isomorphism. 
\end{lemma}
\begin{proof}
The isomorphism result is proven  for  asymptotically hyperbolic manifolds that are conformally compact in \cite[Proposition 3.3]{Lee:1995}  (based on the argument of \cite[Section 3]{Graham-Lee:1991}; see also \cite{Lee:2006} for a general class of operators.) It is clear that the proof can be  adapted for our class of asymptotically hyperbolic manifolds. 
\end{proof}

We also need the following standard Fredholm property for our class of operators. Note similar statements under greater generality can be found in \cite{Lee:2006}, but we include a proof more specific to our setting for completeness.
\begin{proposition}\label{proposition:Fredholm}
Let $(M, g)$ be an  $n$-dimensional asymptotically hyperbolic manifold and $s\in (-1, n)$. Let $T:C^{k,\alpha}_{-s}\to C^{k-2,\alpha}_{-s}$ be asymptotic to $\Delta - n$. Then $T$ is Fredholm. 

\end{proposition}

\begin{proof}
We write $T u = T_0 u + \xi\cdot \nabla u+ (\eta+n) u $. Note $T_0$ is an isomorphism by Lemma~\ref{lemma:isomorphism}, and hence Fredholm. To show that $T$ is Fredholm, it suffices to show that the map $T-T_0 : C^{k,\alpha}_{-s}\to C^{k-2,\alpha}_{-s}$ is compact.

Let $\{ u_i \}$ be a sequence of functions in $C^{k,\alpha}_{-s}$ with $\| u_i \|_{C^{k,\alpha}_{-s}}=1$. We show that $\{ (T-T_0)u_i\}$ has a convergent subsequence in $C^{k-2,\alpha}_{-s}$. By Lemma~\ref{lemma:basic}, $C^{k,\alpha}_{-s} \subset C^{k}_{-s+\epsilon}$ is compact  for $\epsilon>0$,  so there is a subsequence (still denoted by $\{ u_i\}$ without loss of generality) that converges to $u$ in $C^{k}_{-s+\epsilon}$. Observe the sequence $\{ (T - T_0) u_i \}$ converges in $C^{k-2,\alpha}_{-s}$ because
\begin{align*}
	\| (T-T_0 )(u_i-u) \|_{C^{k-2,\alpha}_{-s} } &=\| \xi\cdot \nabla (u_i -u)+ (\eta + n ) (u_i-u) \|_{C^{k-2,\alpha}_{-s} } \\
	& \le C\left[ \| \xi \|_{C^{k-2,\alpha}_{-\epsilon}} \| \nabla (u_i -u) \|_{C^{k-2,\alpha}_{-s+\epsilon}} + \| \eta+n\|_{C^{k-2,\alpha}_{-\epsilon}} \| u_i -u \|_{C^{k-2,\alpha}_{-s+\epsilon}}\right] \\
	&\le C \|u_i - u \|_{C^{k-1,\alpha}_{-s+\epsilon}}\\
	& \le C \|u_i - u \|_{C^{k}_{-s+\epsilon}} \to 0 \qquad \mbox{ as } i\to \infty.
\end{align*}

\end{proof}

\section{Surjectivity of the linearized scalar curvature map} \label{section:static}

Let $(\Omega, g)$ be a  Riemannian manifold.  The linearization $L_g$ of the scalar curvature map at $g$ acts on a symmetric $(0,2)$-tensor $h\in C^2_{\mathrm{loc}}$  by the formula
\begin{align}\label{equation:linearized}
L_g h=- \Delta (\mathrm{tr}\, h) +  \mathrm{div} \, \mathrm{div} \, h - h \cdot \mathrm{Ric}_g,
\end{align}
 and the formal $L^2$-adjoint operator $L_g^*$ is given by, for a function $V\in C^2_{\mathrm{loc}}$,
\begin{align}\label{equation:adjoint}
	L_g^* V= -(\Delta V) g + \nabla^2 V- V\, \mathrm{Ric}_g.
\end{align}
  Here $\mathrm{div}$, $\mathrm{tr}$, $\cdot$, $\Delta$, and $\nabla$  are all taken with respect to $g$. 

We say that $(\Omega, g)$ is \emph{static} if it admits a function $V$,  not identically zero, that satisfies the static equation
\begin{align}\label{equation:static}
	L_g^* V=0.
\end{align} 
We call a solution $V$ to this equation a \emph{static potential}.  Equation \eqref{equation:static}  is equivalent to the following equation:
\begin{align*}
	\nabla^2 V &= \Big(\mathrm{Ric}_g - \tfrac{1}{n-1} R_g \,g\Big) V. 
\end{align*}

\begin{example}
It is well-known that a static manifold has constant scalar curvature on each connected component \cite{Fischer-Marsden:1975}, so a static asymptotically hyperbolic manifold (which is assumed to be connected in Definition~\ref{definition:AH}) must have constant scalar curvature $-n(n-1)$. Thus, \eqref{equation:static} implies 
\begin{align} \label{equation:static-2}
\begin{split}
	\nabla^2 V &= (\mathrm{Ric}_g + ng) V\\
	\Delta V&= n V. 
\end{split}
\end{align}
The prototype of a static asymptotically hyperbolic manifold is hyperbolic space. Recall in Remark~\ref{remark:static},  the space of static potentials is an $(n+1)$-dimensional real vector space spanned by the functions $\sqrt{1+r^2}, x_1, \dots, x_n$ with respect to the coordinates of the hyperboloid model. They come from the restriction of the Minkowski coordinate functions $t, x_1,\dots ,x_n$ to the hyperboloid. 
\end{example}

The goal of this section is to analyze the growth rate of $V$ solving $L_g^*V = \tau $  on an asymptotically hyperbolic manifold $(M, g)$ where $\tau\in C^0_{1-q}$. Specifically, we show  in Theorem~\ref{theorem:expansion} below that such $V$ must either grow linearly in a cone region  or go to zero at infinity. As an application (the case $\tau=0$), at the end of this section we prove Theorem~\ref{theorem:surjectivity}  that $L_g$ is surjective between the appropriate weighted H\"older spaces.  We also use Theorem~\ref{theorem:expansion} (the  case $\tau\neq 0$)  in the proof of Theorem~\ref{theorem:mass-rigidity} in the next section. 

We remark that it is possible to obtain more detailed asymptotics of $V$ (at least for the case $\tau = 0$) as discussed in \cite[Remark A.3]{Chrusciel-Delay:2018} by their analysis. Here we establish elementary properties for a class of inhomogeneous, second-order linear ODEs that suffice for our purpose.

We analyze the asymptotic behavior of a static potential, by studying the static equation along geodesic rays. Note  $\nabla^2 V = (\mathrm{Ric}_g + ng) V = \left(g + O^{0,\alpha}(r^{-q} ) \right)V$ by the asymptotically hyperbolic assumption. The corresponding equation along a geodesic ray is asymptotic to $u'' = u$.  We prove in the next three technical  lemmas that the solutions to a large class of ODEs share similar properties as the solutions to $u'' = u$, which are generated by $e^t, e^{-t}$. 

\begin{lemma}\label{lemma:positivity}
Let $P(t), Q(t) \in C^{0,\alpha}([0,\infty))$  and $Q> 0$. Consider the ODE given by
\begin{align}\label{equation:ODE0}
	u''= Pu'+Qu.
\end{align}
Then the following holds:
\begin{enumerate}
\item A solution $u$ has at most one zero, unless $u$ is identically zero.\label{item:no-zero}
\item If $u$ and $v$ are two solutions satisfying the initial condition $u(0)\ge v(0)$ and $u'(0)\ge v'(0)$, then $u(t)>v(t)$ and $u'(t)>v'(t)$ for all $t>0$, unless $u$ is identical to $v$.  \label{item:comparison}
\item There is a solution $u$ with $u(t)>0$ and $u'(t)<0$, for all $t$.   \label{item:existence}
\end{enumerate}
\end{lemma}
\begin{proof}
Let $K(t) = \exp \big(-\int_0^t P(s)\, ds\big) > 0$. Then 
\begin{align} \label{equation:ODE}
	(Ku')' = KQu.
\end{align}
To see \eqref{item:no-zero}, suppose that $u$ is not identically zero and, to give a contradiction, that $u$ has two or more zeros. Let $t_1<t_2$ be two adjacent zeros. We may without loss of generality assume that $u>0$ on $(t_1, t_2)$. This implies that $u'(t_1)\ge 0$ and $u'(t_2)\le 0$. In fact, both inequalities are strict; otherwise $u$ is identically zero by uniqueness of solutions. However, this contradicts the fact that $Ku'$ is increasing on $[t_1, t_2]$ by \eqref{equation:ODE}. For \eqref{item:comparison}, by linearity it suffices to show that if $u$ is a solution satisfying the initial condition $u(0)\ge 0$ and $u'(0)\ge 0$, then $u(t)>0$ and $u'(t)>0$ for all $t>0$, unless $u$ is identically zero. The desired statement in \eqref{item:comparison} follows from \eqref{equation:ODE} and by observing that if $u\ge 0$ then $Ku'$ is increasing.

We now prove (3) by constructing a compact family of solutions. For an integer $j>0$, let $u_j$ be the solution that satisfies $u_j(0)=1$ and $u_j(j)=0$. By \eqref{item:no-zero} and \eqref{item:comparison}, we have $0\le u_j < u_{j+1} < u_{j+2} < \cdots  <1$ and $u_j' < u_{j+1}' < u'_{j+2}< \cdots <0$ for $t\in (0, j]$. Using \eqref{equation:ODE0} to bound the  higher derivatives, we see that $u_j$ is locally uniformly bounded in $C^{2,\alpha}$. By Arzela-Ascoli, a subsequence locally uniformly converges to a solution $u$ in $C^2([0,\infty))$ that satisfies  $u(0)=1$ and $0\le u(t) \le 1, u'\le 0$ for all $t$. It is straightforward to verify that the inequalities are strict:  $u(t)>0$ and $u'(t)<0$ for all $t$. 

\end{proof}

\begin{lemma}\label{lemma:ODE-estimate}
Let $P(t), Q(t)\in C^{0,\alpha}([0, \infty))$. Suppose  $1+Q> 0$ and  that there are  constants $d, C_0>0$ such that $|P(t)|, |Q(t)| \le C_0 e^{-dt}$. Then there are two linearly independent solutions  $u_1$ and $u_2$ to the homogeneous equation
\[
u'' =P u' + (1+Q)u,
\]
and $u_1, u_2$ satisfy the following:  there is  a constant $C>0$ such that, for all $t$,
\begin{align}\label{equation:fundamental_solution}
\begin{split}
	&C^{-1} e^t \le u_1(t) \le Ce^t, \qquad  \quad C^{-1} e^t \le u_1'(t) \le Ce^t,\\
	&C^{-1} e^{-t} \le u_2(t)  \le Ce^{-t}, \qquad C^{-1} e^{-t} \le -u_2'(t)  \le Ce^{-t}.
\end{split}
\end{align}
\end{lemma}
\begin{proof}
Let $u_1$ be a solution with the initial condition $u_1(0) = 1$ and $u_1'(0)>0$. By (2) in Lemma~\ref{lemma:positivity}, we have $u_1>0$ and $u_1'>0$ for all $t$.  Let $w(t) = u_1(t) + u'_1(t) $. Then $w>0$ satisfies 
\begin{align*}
	w' = (1+P)u_1'+ (1+Q)u_1.
\end{align*}  
This implies the following differential inequality for $w$:
\[
	(1-|P|-|Q|)w \le w' \le (1+|P|+|Q|) w.
\]
Integrating the inequality gives
\[
	w(0)  \exp\left( \int_0^t (1-|P(s)| - |Q(s)|)\, ds\right)\le w(t)  \le w(0) \exp\left( \int_0^t (1+ |P(s)| + |Q(s)|)\, ds\right).
\]
That is, there is a constant $C_1>0$ (depending only on  $w(0), \| P\|_{L^1}$, and $\| Q \|_{L^1}$) such that 
\begin{align} \label{equation:addition}
	C^{-1}_1 e^t \le u_1(t)+u'_1(t)\le  C_1e^{t}.
\end{align}
This gives the upper bound for $u_1, u_1'$ in \eqref{equation:fundamental_solution}. To derive the lower bound for $u_1, u_1'$, we set $z (t) = u_1(t) - u'_1(t)$. Then $z' = - z - Pu'_1 -Qu_1$ and $|z'+z|\le 2C_0C_1 e^{(1-d)t}$. 
Solving the differential inequality gives  $|z|\le C_2 ( e^{(1-d)t} + e^{-t} + te^{-t})$ for some constant $C_2>0$. For $t$ sufficiently large,  we derive $|u_1(t) - u'_1(t)| \le  \frac{1}{2} C_1^{-1} e^{t}$. Together with \eqref{equation:addition}, we obtain the desired estimate  \eqref{equation:fundamental_solution} for $u_1, u_1'$.

By (3) of Lemma~\ref{lemma:positivity}, there is a solution $u_2$ so that $u_2(t)>0$ and $u_2'(t)<0$ for all $t$. Set $h(t) = u_2(t) - u'_2(t)$. Then $h>0$ satisfies
\[
	h' = (1-P )u_2' - (1+Q) u_2,
\]
and hence $(-1-|Q|-|P|) h \le h' \le  (-1+|Q|+|P|) h$. Just as  computing above, we have $C^{-1} e^{-t} \le u_2(t) -u_2'(t) \le Ce^{-t}$, which gives the upper bound for $u_2, u_2'$ in \eqref{equation:fundamental_solution}. Similarly, by estimating the differential inequality for $u_2+u_2'$, we derive the desired lower bound.

Lastly, we note that the two solutions $u_1, u_2$ are linearly independent because their Wronskian is not zero and furthermore, by  \eqref{equation:fundamental_solution}, 
\begin{align}\label{equation:Wronskian}
	\mathrm{det} \begin{bmatrix} u_1 & u_2 \\ u_1' & u_2' \end{bmatrix} = u_1 u_2' - u_2 u_1' \le -2C^{-2} \qquad \mbox{ for all }t.
\end{align}
\end{proof}

\begin{lemma}\label{lemma:inhomogeneous}
 Let $P(t), Q(t)\in C^{0,\alpha}([0,\infty))$ and $f(t) \in C^0([0,\infty))$. Suppose $1+Q>0$ and that there are constants $d, C_0>0$ such that $|P(t)|, |Q(t)|, |f(t)| \le C_0 e^{-dt}$. Let $u$ solve
\begin{align} \label{equation:inhomogeneous}
	u'' =P u' + (1+Q)u +f.
\end{align}
Then there are constants $C>0$ and $c_1, c_2$ such that, for all $t\ge a$, 
\begin{align} \label{equation:particular}
	|u(t)- \big(c_1 u_1(t) + c_2 u_2(t)\big) |\le \left\{\begin{array}{ll}  Ce^{-dt} & \mbox{ for } d\neq 1\\ Cte^{-t}  & \mbox{ for } d=1\end{array}\right..
\end{align}

\end{lemma}
\begin{proof}
Let $u_p$ be a particular solution to \eqref{equation:inhomogeneous}. Notice that $u-u_p$ satisfies the homogeneous equation, and hence is a linear combination of $u_1$ and $u_2$, where $\{u_1, u_2\}$ is the set of fundamental solutions from Lemma~\ref{lemma:ODE-estimate}. It suffices to show that the estimate \eqref{equation:particular} holds for $u_p$.

By the method of variation of parameters, we can choose $u_p$ to be 
\[
	u_p = \alpha_1 u_1 + \alpha_2 u_2, 
\]
where the functions $\alpha_1, \alpha_2$ are defined by
\begin{align*}
	\alpha_1(t) &= - \int_0^t\frac{u_2 (s) f(s)}{u_1(s) u_2'(s) - u_2 (s) u_1'(s) } \, ds\\
	\alpha_2 (t)&= \int_0^t \frac{u_1 (s) f(s)}{u_1(s) u_2'(s) - u_2 (s) u_1'(s) } \, ds.
\end{align*}

If $d> 1$, using  \eqref{equation:fundamental_solution}, \eqref{equation:Wronskian}, and the assumption on $f$, we see that both integrals converge as $t\to \infty$. Let $A_i = \lim_{t\to \infty} \alpha_i(t)$ for $i=1, 2$. There is a constant $C>0$ so that
\begin{align*}
	|\alpha_1 (t) - A_1 | &\le \int_t^\infty \left| \frac{u_2 (s) f(s)}{u_1(s) u_2'(s) - u_2 (s) u_1'(s) } \right|\, ds\le C\int_t^\infty e^{-s} |f(s)|\, ds \le Ce^{-(d+1)t}\\
	|\alpha_2(t) - A_2| & \le \int_t^\infty \left| \frac{u_1 (s) f(s)}{u_1(s) u_2'(s) - u_2 (s) u_1'(s) } \right|\, ds\le C\int_t^\infty e^{s} |f(s)|\, ds \le Ce^{-(d-1)t}.
\end{align*} 
It implies that 
\[
	|u_p - A_1 u_1 - A_2 u_2| \le |\alpha_1 - A_1| u_1 + |\alpha_2 - A_2 | u_2 \le Ce^{-dt}.
\]

If $0<d\le 1$, then $\lim_{t\to \infty} \alpha_2(t)$ may not converge. Nevertheless, there is a constant $C>0$ such that $|\alpha_2| \le Ce^{(1-d)t}$ if $d\neq 1$ and $|\alpha_2|\le Ct$ if $d=1$. Together with the above estimate for $\alpha_1$, we obtain
\[
	|u_p - A_1 u_1| \le |\alpha_1 - A_1| u_1 + |\alpha_2  | u_2 \le \left\{\begin{array}{ll}  Ce^{-dt} & \mbox{ for } d\neq 1\\ Cte^{-t}  & \mbox{ for } d=1\end{array}\right..
\]
\end{proof}

We proceed to discuss the asymptotics of a function that solves the static equation up to an error term. We define a \emph{cone} $U$  as an unbounded open subset  in $M\setminus K$ that consists of points in spherical coordinates such that, for some $r_0>0$ and  a non-empty open subset $\Theta$ in the domain of the angular coordinates on $S^{n-1}$:
\[
	U = \{ (r, \theta_1, \cdots, \theta_{n-1})\in M\setminus K: r> r_0\mbox{ and }(\theta_1, \dots, \theta_{n-1}) \in \Theta\}.
\]
\begin{theorem}\label{theorem:expansion}
Let $(M, g)$ be an asymptotically hyperbolic manifold, and $V\in C^2_{\mathrm{loc}}(M\setminus K)$ satisfy
\begin{align} \label{equation:static-inhomogeneous}
	L_g^* V = \tau
\end{align}
where  $\tau\in C^0_{1-q}(M\setminus K)$ is a symmetric $(0,2)$-tensor. Then $V$ satisfies precisely one of the following:
\begin{enumerate}
\item There is a cone $U\subset M\setminus K$ and a constant $C>0$ such that
\[
	C^{-1} |x| \le |V(x)| \le  C|x| \qquad \mbox{ for all } x\in U.
\]
\item There are constants $C>0$ and $0<d\le 1$  such that 
\[
	|V(x)|\le C|x|^{-d} \qquad \mbox{ for all } x\in M\setminus K.
\]
\end{enumerate}
\end{theorem}

\begin{proof}
Let $B_r$ be a large coordinate ball in $M$ that contains $K$.  It suffices to prove the theorem on $M\setminus B_r$. Note that any point $x\in M\setminus B_r$ could be reached by a geodesic emanating from $\partial B_r$ with the initial velocity $\partial_r$.  Let $\gamma(t)$, $0\le t < \infty$, be the geodesic emanating from a point $p \in \partial B_r$ with $\gamma'(0) = \partial_r$, parametrized by the arc length parameter $t$, i.e.
\[
	t = d_{g} (p, \gamma(t)).
\]
With respect to the hyperbolic metric $b$ on $M\setminus B_r$ (pull back by the diffeomorphism that gives the chart at infinity) and letting $o$ be the origin of $\mathbb{H}^n$, we have $d_b(o, \gamma(t)) = \sinh^{-1}(|\gamma(t)|)$ and hence $|d_b(p, \gamma(t)) - \sinh^{-1}(|\gamma(t)|) |\le d_b(o, p)$ by the triangle inequality, where $|\gamma(t)|$ denotes the radial coordinate of the point $\gamma(t)$. Since the distance in $g$ is comparable to the distance in $b$ by the asymptotically hyperbolic assumption, there is a  constant $C>0$ such that $| t-\sinh^{-1} (|\gamma(t)|)|\le C$ for all $t$. Thus, there is a constant $C>0$ such that  
\begin{align} \label{equation:hyperbolic-distance}
	C^{-1} e^{t} \le |\gamma(t)| \le Ce^{t}.
\end{align}

By~\eqref{equation:static-inhomogeneous} and the assumption on $\tau$, we have
\begin{align}\label{equation:hessian}
\begin{split}
	\nabla^2 V &=\left(\mathrm{Ric}_g-\tfrac{1}{n-1} R_g\, g\right) V +\tau- (\tfrac{1}{n-1} \mathrm{tr}\, \tau)\,g\\
	&= \left( g+ O^{0,\alpha}(r^{-q}) \right)V + O(r^{1-q}).
\end{split}
\end{align}
Let $u(t) = V\circ \gamma(t)$. The equation \eqref{equation:hessian} implies that $u$ satisfies the following ODE:
\begin{align*}
	u'' &= \nabla^2 V(\gamma'(t), \gamma'(t)) + \nabla V(\nabla_{\gamma'(t)} \gamma'(t))\\
	&= \nabla^2 V(\gamma'(t), \gamma'(t))\\
	&= (1+Q(t)) u + f,
\end{align*}
where $|Q(t)|\le Ce^{-qt}$ and $|f(t)|\le Ce^{(1-q)t}$ by \eqref{equation:hessian} and \eqref{equation:hyperbolic-distance}. By Lemma~\ref{lemma:inhomogeneous}, there is a constant $C>0$ and $d\in (0,1]$ such that $V$ satisfies 
\begin{enumerate}
\item either $C^{-1}e^{t}\le |V(\gamma(t))|\le Ce^t$  for all $t$
\item or  $|V(\gamma(t))|\le Ce^{-dt}$  for all $t$.  
\end{enumerate}
If (1) holds for some geodesic $\gamma$, by continuous dependence of ODE solutions on the initial conditions, the estimate $C^{-1}|x| \le |V(x)| \le C|x|$  holds in a cone, where we use \eqref{equation:hyperbolic-distance} to replace $e^t$ with $|\gamma(t)|$ and enlarge the constant $C$ if needed.  If (1) does not hold for any geodesic $\gamma$, then (2) holds for all $\gamma(t)$, with a uniform constant $C$ by compactness of $\partial B_r$.
Using \eqref{equation:hyperbolic-distance} and enlarging $C$ if necessary, we have $|V(x)|\le C|x|^{-d}$ for all $x\in M\setminus K$.
\end{proof}

\begin{corollary}\label{corollary:asymptotics}
Let $(M, g)$ be an asymptotically hyperbolic manifold, and $V\in C^2_{\mathrm{loc}}$ solve $L_g^* V=0$ in $M$. If $V$ is not identically zero, then there is a cone $U\subset M\setminus K$ and a constant $C>0$ such that $V$ satisfies
\[
	C^{-1} |x| \le |V(x)| \le  C|x| \qquad \mbox{ for all } x\in U.
\]
\end{corollary}
\begin{proof}
Recall $\Delta V = nV$ in \eqref{equation:static-2}. By letting $\tau=0$ in Theorem~\ref{theorem:expansion}, we have that either the desired estimate holds, or there are constants $d, C>0$ such that  $|V(x)|\le C|x|^{-d}$ for all $x\in M\setminus K$. However, the latter case implies that $V$ is identically zero by  maximum principle.
\end{proof}

We now prove the main result in this section. 

\begin{proof}[Proof of Theorem~\ref{theorem:surjectivity}]
It suffices to show that the linearized scalar curvature map is surjective.  Local surjectivity of the scalar curvature map follows from standard functional analysis. 

We first show that the range of $L_g$ is closed. Define the operator $T(u) := L_g (ug)$ for functions $u\in C^{k,\alpha}_{-s}(M)$. Then $\frac{1}{1-n} T(u)=\Delta u +\tfrac{1}{n-1}R_gu$ is asymptotic to $\Delta - n$ and hence Fredholm by Proposition~\ref{proposition:Fredholm}. In particular, the range of $T$ has finite codimension, and so does the range of $L_g$. It implies that the range of $L_g $ is closed.

To see surjectivity of $L_g$, we show that the adjoint operator $L_g^*: (C^{k-2, \alpha}_{-s})^* \to (C^{k,\alpha}_{-s})^*$ has a trivial kernel. Let $u\in  (C^{k-2, \alpha}_{-s})^*$ weakly solve $L_g^*u=0$. Note that since $C_c^\infty$ is dense in $C^{k-2,\alpha}_{-s}$, $u$ is, in  particular, a distribution. Taking the trace of $L_g^* u=0$ implies that $u$ weakly solves an elliptic PDE, whose coefficients are locally smooth by the hypothesis $g\in C^\infty_\mathrm{loc}$. Applying elliptic regularity for distribution solutions (see, e.g. \cite[Theorem 6.33]{Folland:1995}), we have $u\in C^{k,\alpha}_{\mathrm{loc}}$ with the duality given by 
\begin{align} \label{equation:pairing}
	 u (\phi) = \int_M u \phi \, d\mu_g,\quad \mbox{ for all } \phi \in C^\infty_c.
\end{align}

Suppose, to give a contradiction, that $u$ is not identically zero. We shall show that the above pairing is not bounded for some $\phi \in C^{k-2,\alpha}_{-s}$.   By Corollary~\ref{corollary:asymptotics}, there is a constant $C>0$ such that $|u(x)|\ge C|x|$ in a nonempty cone $U\subset M\setminus K$. We may without loss of generality assume $u>0$ and hence $u(x) \ge C|x|$ on $U$.  Let $\phi(x)$ be a non-negative function in  $C^{k-2,\alpha}_{-s}$ so that $\phi(x) = |x|^{-s}$ in  a smaller cone $U'\subset U \subset M\setminus K$ and $\phi\equiv 0$ outside $U$.  Let $\phi_i \in C_c^\infty(U)$ be a monotone sequence of non-negative functions that converge to $\phi$ in $C^{k-2,\alpha}_{-s}$ (for example, let $\phi_i = \chi_i \phi$ where $\chi_i$ is a monotone sequence of bump functions  uniformly bounded in $C^\infty$). Then 
\begin{align*}
	u( \phi ) &= \lim_{i\to \infty} u (\phi_i)=\lim_{i\to \infty}\int_M u \phi_i \, d\mu_g = \int_M u \phi \, d\mu_g, 
\end{align*}
where the first equality is from continuity of $u$ as a functional, the second equality is by~\eqref{equation:pairing}, and the last equality is by monotone convergence theorem. However, since $s\le  n$ and $d\mu_g =\Big( \frac{r^{n-1}}{\sqrt{1+r^2}} + O(r^{n-2-q})\Big) \, dr d\omega $, the last integral diverges to infinity:
\[
	\int_M \phi u\, d\mu_g \ge C\int_{U'} r^{1-s}\, d\mu_g = \infty.
\]

\end{proof}

\section{Mass minimizer and static uniqueness} \label{section:main-theorem}

Let $(M,g)$ be an $n$-dimensional asymptotically hyperbolic manifold. Consider the following Banach (affine) space of symmetric $(0,2)$-tensors: 
\begin{align} \label{equation:Banach}
\begin{split}
\mathcal{B} &= \{ g+ h : h\in C^{2,\alpha}_{-q}(M) \mbox{ is a symmetric $(0,2)$-tensor}\}\\
\mathcal{M}\subset \mathcal{B} &\mbox{ is an open neighborhood of $g$ containing positive definite tensors}.
\end{split}
\end{align}

Suppose $f\in C^{2,\alpha}_{\mathrm{loc}}(M)$ satisfies the following asymptotics, for some $a_0, a_1, \dots, a_n\in \mathbb{R}$, 
\begin{align} \label{equation:static-asymptotics}
	f (x)= a_{0}\sqrt{1+r^2} -\left( a_{1} x_1 + \cdots + a_{n} x_n \right) + O^{2,\alpha}(|x|^{1-q}).
\end{align}
 By direct computation,
\begin{align} \label{equation:hessian-static}
\begin{split}
	\nabla^2 \sqrt{1+r^2} &=  \sqrt{1+r^2}\, g+ O^{0,\alpha}(r^{1-q})\\
	\nabla^2 x_i &= x_i \, g + O^{0,\alpha}(r^{1-q})\quad \mbox{ for } i =1,\dots, n.
\end{split}
\end{align}
Therefore, we have
\begin{align} \label{equation:almost-static}
	L_g^* f =  -( \Delta f) g +\nabla^2 f - f\mathrm{Ric}_g = O^{0,\alpha}(r^{1-q}).
\end{align}
We define  the corresponding functional $\mathcal{F}$ on $\mathcal{M}$ by
\begin{align}\label{equation:functional}
	\mathcal{F}(\gamma) &=a_0 p_0(\gamma)-\left( a_1p_1(\gamma)+\cdots +a_n p_n(\gamma)\right)  - \int_M \left(R(\gamma) + n(n-1)\right) f\, d\mu_g
\end{align}
where $R: \mathcal{M}\to C^{0,\alpha}_{-q}$ is the scalar curvature map and recall $\left(p_0(\gamma), \dots, p_n(\gamma) \right)$ denotes the mass of~$\gamma$.

It may not be immediately obvious that $\mathcal{F}(\gamma)$ is finite for $\gamma\in \mathcal{M}$. Since $\gamma$ is not assumed to satisfy the scalar curvature assumption \eqref{item:scalar} of Definition~\ref{definition:AH}, either term  in the definition of $\mathcal{F}$ may diverge. In the next lemma, we give an alternative expression for  $\mathcal{F}$ and show that  $\mathcal{F}$  is well-defined. We also compute its first variation. 
\begin{lemma}\label{lemma:functional}
Let  $f\in C^{2,\alpha}_{\mathrm{loc}}(M)$ satisfy the asymptotics
\[
	f (x)= a_{0}\sqrt{1+r^2} -\left( a_{1} x_1 + \cdots + a_{n} x_n \right) + O^{2,\alpha}(|x|^{1-q}).
\]
Then the corresponding functional $\mathcal{F}: \mathcal{M}\to \mathbb{R}$ can be expressed as 
\begin{align} \label{equation:functional2}
\begin{split}
\mathcal{F}(\gamma)&=\int_M \left(\left[ L_g (\gamma-\mathfrak{b}) - (R(\gamma) + n(n-1) )\right]f-(\gamma - \mathfrak{b})\cdot L_g^*f\right)  \, d\mu_g,
\end{split}
 \end{align}
where $\mathfrak{b}$ is any fixed smooth symmetric $(0,2)$-tensor in $M$ that coincides with the hyperbolic metric~$b$ in the chart at infinity.

 As a consequence, the linearization $D\mathcal{F}|_g : C^{2,\alpha}_{-q} \to \mathbb{R}$  at $g$  is given by 
\[
	D\mathcal{F}|_g(h) = -\int_M h \cdot L_g^*f \, d\mu_g.
\]
\end{lemma}

\begin{proof}
We recall the formulas for $L_g $ and $L_g^*$ in \eqref{equation:linearized}  and \eqref{equation:adjoint} for the following computations.

Let $e = \gamma -\mathfrak{b}$. By  Definition \ref{def:massfunctional} and  Remark~\ref{remark:mass}, we have
\begin{align*}
	\mathcal{F}(\gamma) &= \lim_{r\rightarrow\infty}\int_{S_r} \big(f\left(\mathrm{div} \, e-d (\mathrm{tr}\, e) \right)(\nu) + (\mathrm{tr} \,e) \,df (\nu) - e(\nabla f, \nu)\big) \, d\sigma_g- \int_M \big(R(\gamma) + n(n-1)\big) f\, d\mu_g \notag\\
	&= \int_M \mathrm{div} \left[f\left(\mathrm{div}\, e-d (\mathrm{tr}\, e) \right) + (\mathrm{tr}\, e) \,df - e(\nabla f,\cdot)\right] \, d\mu_g - \int_M \big(R(\gamma) + n(n-1)\big) f\, d\mu_g \notag\\
	&= \int_M \left[ \mathrm{div} \,\mathrm{div}\, e - \Delta  (\mathrm{tr} \, e) - R(\gamma) - n(n-1) \right]f\, d\mu_g -\int_M \big( - (\Delta f) g +\nabla^2  f\big) \cdot e \, d\mu_g\\
	&= \int_M \left[ L_g (e) - (R(\gamma) + n(n-1) )\right]f\, d\mu_g -\int_M \big( - (\Delta f) g +\nabla^2  f- f\mathrm{Ric}_g\big) \cdot e \, d\mu_g.
\end{align*} 
Note  \eqref{equation:almost-static} and  $R(\gamma)+n(n-1) = L_g(e) + O(r^{-2q})$ by Taylor expansion.  Both integrals converge by routine computations.
\end{proof}

So far, we have considered the functional $\mathcal{F}$ defined by an arbitrary function $f$ satisfying the asymptotics \eqref{equation:static-asymptotics}. In what follows, we will choose specifically $f$ which is an eigenfunction  $\Delta f = nf$.

\begin{lemma}[{\cite[Lemma 3.3]{Qing:2003}}]\label{lemma:eigenfunction}
Let $(M, g)$ be an asymptotically hyperbolic manifold.  There are functions $f_0, f_1, \dots, f_n\in C^{2,\alpha}_{\mathrm{loc}} (M)$ satisfying $\Delta f_0 = nf_0$ and $\Delta f_i=nf_i$ for $i=1,\dots, n$ with the asymptotics
\begin{align*} 
	f_0 (x) &= \sqrt{1+r^2} + O^{2,\alpha}(r^{1-q})\\
	f_i (x) &= x_i + O^{2,\alpha}(r^{1-q}).  
\end{align*}
\end{lemma}
\begin{proof}
Taking the trace of equations in~\eqref{equation:hessian-static} yields 
\begin{align*}
	\Delta \sqrt{1+r^2} &= n \sqrt{1+r^2} + O^{0,\alpha}(r^{1-q}) \\
	\Delta x_i &= n x_i + O^{0,\alpha}(r^{1-q}).
\end{align*}
Note that the operator $\Delta - n :C^{2,\alpha}_{1-q}\to C^{0,\alpha}_{1-q}$ is an isomorphism by Lemma~\ref{lemma:isomorphism}. There is a unique $v\in C^{2,\alpha}_{1-q}$ that solves $\Delta v -nv= -\Delta \sqrt{1+r^2} + n \sqrt{1+r^2}$.  We set $f_0 = \sqrt{1+r^2} + v$. Other eigenfunctions $f_i$ are obtained similarly.
\end{proof}

\begin{theorem}\label{theorem:mass-rigidity}
Let $(M,g)$ be an asymptotically hyperbolic manifold with scalar curvature $R_g \ge -n(n-1)$ and with the equality $p_0 = \sqrt{p_{1}^2 + \cdots + p_{n}^2}$, where $(p_0, p_1, \dots, p_n)$ is the mass of $g$.   Suppose the following holds:
\begin{itemize}  
\item[($\star$)] There is an open neighborhood $\mathcal{M}$ of $g$ in $\mathcal{B}$ such that for any $\gamma\in \mathcal{M}$ with $R(\gamma) = R_g$, the inequality $p_0(\gamma)\ge \sqrt{ (p_1(\gamma))^2 + \dots + (p_n(\gamma))^2}$ holds.   
\end{itemize}
 Then $(M, g)$ is static with a static potential $f>0$ satisfying the asymptotics:
 \begin{align}\label{equation:static-potential}
	f &=\begin{cases} p_0\sqrt{1+r^2}-(p_1 x_1  + \cdots + p_n x_n) + O^{2,\alpha}(r^{1-q}) & \mbox{ if } p_0>0\\  \sqrt{1+r^2}+ O^{2,\alpha}(r^{1-q})  & \mbox{ if } p_0 = 0\end{cases}.
\end{align}
\end{theorem}
\begin{proof}

\noindent{\bf Case 1: $p_0>0$.} Let $f_0, f_1, \dots, f_n$ be from  Lemma~\ref{lemma:eigenfunction}. Define
\[
	f= p_0 f_0 - (p_1 f_1 + \dots + p_n f_n), 
\]
where $(p_0, p_1, \dots, p_n)$ is the mass of $g$. Note $\Delta f = n f$.  Since $f>0$ outside a large compact set, it follows from the maximum principle that $f$ is everywhere positive. We claim that $f$ is a static potential on $M$.

Consider the functional $\mathcal{F}:\mathcal{M}\to \mathbb{R}$ defined by \eqref{equation:functional} corresponding to this particular choice of $f$ with the coefficients $a_k = p_k$ for all $k=0, 1, \dots, n$. Let $R:\mathcal{M} \to C^{0,\alpha}_{-q}$ be the scalar curvature map that sends $\gamma$ to the scalar curvature of $\gamma$. Define 
$\mathcal{C}_g = \{ \gamma\in \mathcal{M}: R(\gamma) = R_g \}$.   By hypothesis~($\star$), for $\gamma\in \mathcal{C}_g$, we have 
\[
	p_0(\gamma) \ge \sqrt{\big(p_1(\gamma)\big)^2+\cdots + \big(p_n(\gamma)\big)^2}.
\]
We compute that the functional $\mathcal{F}$ achieves a local minimum  at $g$ among the constraint set $\mathcal{C}_g$:
\begin{align*}
\mathcal{F}(\gamma) - \mathcal{F}(g) &=p_{0} p_{0} (\gamma) - \big( p_{1}  p_{1}(\gamma) + \cdots + p_{n} p_{n}(\gamma) \big) \\
 &\ge p_{0}  p_{0} (\gamma)-\sqrt{p_1^2 + \cdots + p_n^2 } \sqrt{\big(p_1(\gamma)\big)^2+\cdots + \big(p_n(\gamma)\big)^2} \\
 &=  p_0 \left(p_{0} (\gamma)- \sqrt{(p_1(\gamma))^2+\cdots + (p_n(\gamma))^2}\right) \\
 &\ge 0 
 \end{align*}
 with equalities realized at $\gamma = g$. 
 
By Theorem \ref{theorem:surjectivity}, $L_g:C^{2,\alpha}_{-q}\to C^{0,\alpha}_{-q}$ is surjective, so we can apply  the method of Lagrange Multipliers (see, for example, ~\cite[Theorem C.1]{Huang-Lee:2019}) to obtain  $\lambda \in (C^{0,\alpha}_{-q})^*$ that satisfies
 \[
 	D\mathcal{F}|_g(h) =\lambda (L_g(h) ) \qquad \mbox{for all $h\in C^{2,\alpha}_{-q}$}.
 \] 
 We substitute the left-hand side above by the first variation formula in Lemma~\ref{lemma:functional} and get
 \begin{align} \label{equation:lambda}
	-\int_M h\cdot L_g^*(f)\, d\mu_g=\lambda (L_g(h) ) \qquad \mbox{ for all } h\in C^{2,\alpha}_{-q}.
\end{align}
Considering $h\in C_c^\infty$ in the above identity implies that $\lambda$, as a distribution,  is a weak solution to $-L_g^* f = L_g^* \lambda$. Taking the trace of the previous equation implies that  $\lambda$ weakly solves an elliptic PDE with locally smooth coefficients, by the hypothesis $g\in C^\infty_{\mathrm{loc}}$. By elliptic interior regularity for distribution solutions (see, for example,~\cite[Theorem 6.33]{Folland:1995}), $\lambda\in C^{2,\alpha}_{\mathrm{loc}}(M)$ with the duality given by
\[ 
	 \lambda (L_g(h) )= \int_M \lambda  L_g(h) \,d\mu_g\qquad \mbox{ for }h\in C^\infty_c(M).
\] 
Together with \eqref{equation:lambda},  $\lambda$ solves $L_g^* \lambda =-L_g^* f$ in the classical sense. 

We recall $L_g^* f\in C^{0,\alpha}_{1-q}$. Applying Theorem~\ref{theorem:expansion} yields that there are numbers $d, C>0$ such that either $|\lambda(x)|\ge C|x|$ in a nonempty cone $U\subset M\setminus K$, or $|\lambda(x)|\le C|x|^{-d}$ in $M\setminus K$.  Since $\lambda$ is a bounded functional on $C^{0,\alpha}_{-q}$, the first case does not occur, by  the same argument as in the last paragraph in the proof of Theorem~\ref{theorem:surjectivity}.  Therefore, we must have $|\lambda(x)|\le C|x|^{-d}$ in $M\setminus K$; in particular, $\lambda(x)\to 0$ as $|x|\to \infty$. Taking the trace of  $L_g^* \lambda =-L_g^* f$  gives that 
\[
	\Delta \lambda - n \lambda = -(\Delta f - nf) = 0.
\]
We conclude $\lambda$ is identically zero by the maximum principle. We conclude that $f$ is a static potential. \\

\noindent{\bf Case 2: $p_0=0$.} We let $f = f_0$ where $f_0$ is from  Lemma~\ref{lemma:eigenfunction}. That is, $f=\sqrt{1+r^2}+ O^{2,\alpha}(r^{1-q})$ and $\Delta f = nf$. Note $f>0$ by maximum principle.  We will show that $f$ satisfies the static equation. Let $\mathcal{F}:\mathcal{M}\to \mathbb{R}$ be the functional defined by \eqref{equation:functional} corresponding to this particular choice of $f$ with $a_0 = 1$ and $a_1=\cdots = a_n=0$. Specifically, 
\[
	\mathcal{F}(\gamma) = p_0(\gamma)  - \int_M \left(R(\gamma) + n(n-1)\right) f\, d\mu_g.
\]
Recall  $\mathcal{C}_g$ defined above.  Among the constraint $\gamma\in \mathcal{C}_g$, we have $\mathcal{F}(\gamma) - \mathcal{F}(g) = p_0(\gamma)-p_0(g)\ge 0$  by hypothesis ($\star$) and thus $\mathcal{F}$ attains the minimum at $\gamma =g$. Now, we can apply the method of the Lagrange multipliers and argue that $f$ is a static potential  as above. 

\end{proof}

We have shown that a metric $g$ that locally minimizes the functional $\mathcal{F}$ possesses a static potential with specific asymptotics. To conclude the proof of Theorem~\ref{theorem:main}, we establish static uniqueness and show isometry to hyperbolic space. (In particular, the case $p_0>0$ in \eqref{equation:static-potential} cannot happen.) 

\begin{lemma}\label{lemma:integral}
Let $(M, g)$ be an asymptotically hyperbolic manifold that admits a positive static potential $f$ with the asymptotics \eqref{equation:static-potential}. Then on any large coordinate ball $B_r$, the following identity holds
\begin{align*}	
	\int_{B_r} f^{-1} |\nabla^2 f-fg|^2\, d\mu_g=  \int_{\partial B_r} \big(\mathrm{Ric}_g  + (n-1) g \big) (\nabla f, \nu) \, d\sigma_g
\end{align*}
where $|\cdot |$ is the norm taken with respect to $g$ and $\nu$ is the outward unit normal vector on $\partial B_r$.
\end{lemma}
\begin{proof}
The following identity is due to X. Wang \cite{Wang.X:2005}. Set $S= \mathrm{Ric}_g + (n-1)g$. By the static equation, $S= 	f^{-1}\nabla^2 f-g$ and $S$ is both trace and divergence free. We compute
\begin{align*}
f^{-1}|\nabla^2 f-fg|^2&=f|S|^2\\
&=fg( f^{-1}\nabla^2 f,S) \quad \mbox{ ($S$ is trace-free)}\\
&=g( \nabla^2 f,S)\\
&=\text{div}(S(\nabla f)) \quad \mbox{ ($S$ is divergence-free)}.
\end{align*}
The lemma follows by integrating the identity on $B_r$ and applying the divergence theorem.
\end{proof}

We are ready to prove Theorem~\ref{theorem:main}. We restate the assumption ($\star$) using the precise Banach spaces defined earlier in \eqref{equation:Banach}.

\begin{main-theorem}
Let $n\ge 3$ and $(M,g)$ be an $n$-dimensional asymptotically hyperbolic manifold with scalar curvature $R_g \ge -n(n-1)$ and with the equality $p_0 = \sqrt{p_{1}^2 + \cdots + p_{n}^2}$, where $(p_0, p_1, \dots, p_n)$ is the mass of $g$.   Suppose the following holds:
\begin{itemize}  
\item[($\star$)] There is an open neighborhood $\mathcal{M}$ of $g$ in $\mathcal{B}$ such that any $\gamma\in \mathcal{M}$ with $R(\gamma) = R_g$, the inequality $p_0(\gamma)\ge \sqrt{ (p_1(\gamma))^2 + \dots + (p_n(\gamma))^2}$ holds. 
\end{itemize}
 Then $(M, g)$ is isometric to hyperbolic space.
 \end{main-theorem}
\begin{proof}
By Theorem~\ref{theorem:mass-rigidity}, $M$ admits a positive static potential $f$ of asymptotics \eqref{equation:static-potential}.  Using Lemma~\ref{lemma:integral} and Proposition~\ref{proposition:mass}, in  either the case $p_0 >0$ or $p_0 =0$, we have the following identity
\begin{align*}
\int_{M} f^{-1} |\nabla^2 f-fg|^2\,d\mu_g&=  \lim_{r\to \infty}\int_{\partial B_r} \Big(\mathrm{Ric}_g  + (n-1) g \big) (\mathring{\nabla} f, \nu_0) \, d\sigma_0= -\tfrac{n-2}{2} H(f)=0.
\end{align*}
This implies $\nabla^2 f = fg$, which characterizes hyperbolic space by an elementary argument, which we present in Proposition~\ref{proposition:rigidity} below. 

Alternatively, we could use again that $f$ satisfies the static equation by Theorem~\ref{theorem:mass-rigidity} to see that $g$ is Einstein with $\mathrm{Ric}_g = -(n-1)g$. Then $M$ is isometric to hyperbolic space by Bishop-Gromov volume comparison.
\end{proof}

\begin{proposition}\label{proposition:rigidity}
Let $(M, g)$ be asymptotically hyperbolic. If there is a non-zero function $f\in C^2_{\mathrm{loc}} (M)$ satisfying $f> -c $ for some real number $c$ and the following equation on $M$:
\begin{align}\label{equation:hessian-f}
	\nabla^2 f = fg, 
\end{align}
then $(M, g)$ is isometric to hyperbolic space. 
\end{proposition}
\begin{proof}
If $f$ has at least one critical point, the result is classical (see  \cite{Tashiro:1965} and also \cite[Theorem C]{Kanai:1983} and \cite[Lemma 3.3]{Qing:2004}).

We now assume that $f$ has no critical point in $M$, i.e. $\nabla f$ is never zero.  We compute the first and second covariant derivatives of $\nabla^2 f = fg$ at a point $p\in M$ with respect to a geodesic normal coordinate chart:
\begin{align*}
	0&= f_{;ijk} - f_{;ikj} - R_{kj\ell i} f^\ell=f_k g_{ij} - f_j g_{ik} -R_{kj\ell i} f^\ell\\
	0&= f(g_{km} g_{ij} - g_{jm} g_{ik} ) - R_{kj\ell i;m} f^\ell - R_{kjmi} f,
\end{align*}	
where  $R_{kj\ell i} = g(\nabla_{\partial_k} \nabla_{\partial_j } \partial_{\ell}-\nabla_{\partial_j} \nabla_{\partial_k } \partial_{\ell} , \partial_i )$ in our convention.  We than obtain the following formulas, for any vector fields $W, X, Y, Z$, 
\begin{align} \label{equation:curvature}
\begin{split}
	R(X, Y, \nabla f, Z) &= g(\nabla f, X) g(Y, Z) - g(\nabla f, Y) g(X, Z)\\
	(\nabla_Z R) (X, Y, \nabla f, W) &= -f\Big( R(X, Y, Z, W) - g(X, Z) g(Y, W) + g(Y, Z) g(X, W)\Big).
\end{split}
\end{align}

Let $\gamma:(-\infty,\infty) \to M$ be the integral  curve of $\frac{\nabla f}{|\nabla f|}$ through a point $p\in M$, i.e. $\gamma'(t) =\frac{\nabla f(\gamma(t))}{|\nabla f|(\gamma(t))}$. By direct computation using \eqref{equation:hessian-f}, we have  $\nabla_{\gamma'} \gamma'=0$  and thus $\gamma$ is a geodesic parametrized by arc length. We compute
\begin{align}
	\frac{d}{dt} f(\gamma(t)) &= g( \nabla f, \gamma'(t) ) = |\nabla f (\gamma(t))| >0 \label{equation:f'}\\
	\frac{d^2} {dt^2} f(\gamma(t))&=\nabla^2 f(\gamma'(t), \gamma'(t) ) = f(\gamma(t)). \notag
\end{align}
Solving the ODE yields that, for $t\in (-\infty, \infty)$,  
\begin{align}\label{equation:solution-f}
f(\gamma(t)) = C_1e^t + C_2 e^{-t},
\end{align} 
 where $C_1 \ge 0 \ge C_2$ and $C_1, C_2$ are not both zero. 

Let $X, Y$ be two  orthonormal vector fields perpendicular to $\gamma'$ and parallel along $\gamma$. 
The sectional curvature $K(X\wedge \gamma') = R(X, \gamma', \gamma', X) = -1$  along $\gamma$ by \eqref{equation:curvature}. Next, we compute  that the sectional curvature $K(t):=K(X\wedge Y)=R(X, Y, Y, X)$ along $\gamma(t)$. In what follows, we slightly abuse the notation and denote $f(t) = f(\gamma(t))$ and $|\nabla f|(t) = |\nabla f|(\gamma(t))$. We compute, for all $t\in (-\infty, \infty)$,  
 \begin{align*} 
  \begin{split}
 	K'(t) &=\gamma' (R(X, Y, Y, X) ) =( \nabla_{\gamma'} R) (X, Y, Y, X) \\
	&= -( \nabla_Y R)(X, Y, X, \gamma') - (\nabla_X R)(X,  Y, \gamma', Y) \qquad  \mbox{(by the second Bianchi identity)}\\
	&=-2 \frac{f(t)}{ |\nabla f|(t)} (K(t)+ 1)
\end{split}
 \end{align*}
 where in the last equation we use the second equation in \eqref{equation:curvature}. We would like to show that $K(t) +1 \equiv 0$ for all $t$. Suppose, to give a contradiction, that $K(t) +1 $ is not identically zero. Then $K(t)+1$ has no zeros, and we can divide the equation of $K'(t)$ by $K(t) +1$ to achieve
 \begin{align}\label{equation:curvature-ODE}
 		\frac{d}{dt} \log |K(t) + 1 |  = -2 \frac{f(t)}{ |\nabla f|(t)}.
 \end{align}
Note that  $ \frac{f(t)}{ |\nabla f|(t)}$ satisfies the following ODE on $\gamma(t)$ by direct computation:
 \begin{align*}
	\frac{d}{dt} \left( \frac{f(t)}{|\nabla f| (t) } \right) = 1 - \left(\frac{f(t)}{|\nabla f|(t)}\right)^2 \quad \mbox{ for } t\in (-\infty, \infty).
\end{align*}
Solving this ODE yields that either (1) $\frac{|f|}{|\nabla f|} \equiv 1$ on $\gamma(t)$, or (2) there is a constant $C> 0$ such that $\frac{f}{|\nabla f|} = 1-\frac{2}{Ce^{2t} + 1}$  on $\gamma(t)$.  For Case (1),  we find $K(t) + 1$ to be either $Be^{-2t}$ or $Be^{2t}$ for some nonzero constant $B$ for $t\in (-\infty, \infty)$ by \eqref{equation:curvature-ODE}, which contradicts the asymptotically hyperbolic assumption.  For Case $(2)$,  $f(t)$ has a zero and hence $C_2 <0$ in \eqref{equation:solution-f}, which contradicts the assumption $f > -c$.


 Varying $p\in M$, we conclude that the sectional curvature of $M$ is identically $-1$, which implies  the universal cover of $M$ is hyperbolic space. Together with the asymptotically hyperbolic assumption, $M$ is isometric to hyperbolic space. 
\end{proof}

\begin{remark}
We thank Piotr Chru\'sciel for pointing out an example that demonstrates the
necessity of the hypothesis $f\ge -c$.  Let $(\Sigma, h)$ be a complete $(n-1)$-dimensional manifold (either closed or unbounded) with bounded sectional curvature. Consider the  product $M = (-\infty, \infty) \times \Sigma$ endowed with the warped product metric $g = dt^2 + (\cosh t)^2 h$. One can directly check that  the sectional curvature of $(M, g)$ approaches $-1$ as $t\to \pm \infty$ and that $f(t) = \sinh t$ satisfies $\nabla^2 f = fg$. (In particular, $\frac{f}{|\nabla f|} = 1-\frac{2}{e^{2t} + 1}$ realizes Case (2) above.) If we further specify $\Sigma = \mathbb{R}^{n-1}$ endowed with a metric $h$ whose sectional curvature is identically $-1$ outside a compact set of $\Sigma$, the sectional curvature of the resulting metric $g$  approaches $-1$ toward the infinity of $M$. However, $(M, g)$ is of constant sectional curvature $-1$ if and only if $h$ is of constant sectional curvature $-1$ everywhere on $\Sigma$.  
\end{remark}


\bibliographystyle{amsplain}
\bibliography{2018}

\end{document}